\documentclass[11pt]{amsart}
\usepackage[utf8]{inputenc}
\usepackage{amsmath}
\usepackage{amsfonts}
\usepackage{amssymb}
\usepackage{amsthm}
\usepackage{csquotes}
\usepackage[english]{babel}
\usepackage[shortlabels]{enumitem}
\usepackage{lmodern}
\usepackage{cite}

\usepackage{color}

\newtheorem{theorem}{Theorem}[section]

\newtheorem{lemma}[theorem]{Lemma}
\newtheorem{proposition}[theorem]{Proposition}

\theoremstyle{definition}

\theoremstyle{remark} \theoremstyle{remark}
\newtheorem{remark}[theorem]{Remark}

\newtheorem*{theorem*}{Theorem}

\DeclareMathOperator{\im}{im}
\DeclareMathOperator{\id}{id}

\newcommand{\summ}{\sum_{\alpha=1}^{s}}

\newcommand{\etaa}{\eta_{\alpha}}

\newcommand{\xia}{\xi_{\alpha}}

\newcommand{\cF}{{\mathcal F}}

\newcommand{\R}{{\mathbb R}}
\newcommand{\Z}{{\mathbb Z}}

\numberwithin{equation}{section}

\begin{document}

\title{How to construct all metric $f$-$K$-contact manifolds}
 
 \author{Oliver Goertsches}
  \address{Philipps-Universit\"at Marburg, Fachbereich Mathematik und Informatik, Hans-Meerwein-Straße 6, 35043 Marburg, Germany}
  \email{goertsch@mathematik.uni-marburg.de}
  
 \author{Eugenia Loiudice}
 \email{loiudice@mathematik.uni-marburg.de}

 \begin{abstract}
We show that any compact metric $f$-$K$-contact, respectively $S$-manifold is obtained from a compact $K$-contact, respectively Sasakian manifold by an iteration of constructions of mapping tori, rotations, and type II deformations.
 \end{abstract}

\keywords{metric $f$-$K$-contact manifolds, $S$-manifolds, mapping torus, deformations, basic cohomology, foliations}

\maketitle
\section{Introduction} 

The class of metric $f$-$K$-contact manifolds generalizes the class of ($K$-) contact manifolds. They are $f$-manifolds (namely smooth manifolds together with a $(1,1)$-tensor $f$ of constant rank and such that $f^3+f=0$, \cite{yano}) endowed with a Riemannian metric, $s$ (Killing) vector fields $\xi_i$ and $s$ one-forms dual to the $\xi_i$'s, satisfying some compatibility conditions (see Section~\ref{sec:preliminaries}).

In addition to generalizing almost complex and almost contact manifolds (which we have in case the $f$-structure $f$ has maximal rank and the dimension of the manifold is even or respectively odd), $f$-manifolds appear naturally when studying the hypersurfaces of almost contact manifolds, see \cite{blair-ludden} for details.  

In the present paper we look at metric $f$-contact manifolds from a dual perspective: motivated by the unusual property of metric $f$-($K$-)contact and $S$-manifolds that their geometric structure is inherited by mapping tori with respect to automorphisms of the structure (see \cite{gl}), we ask whether a given metric $f$-$K$-contact or $S$-manifold is the mapping torus of a lower-dimensional such manifold. The main result of this paper reads:
\begin{theorem}
Any compact connected metric $f$-$K$-contact (resp.\ $S$-) manifold is obtained from a compact $K$-contact (resp.\ Sasakian) manifold by a finite iteration of the following operations:
\begin{enumerate}
\item construction of the mapping torus with respect to an automorphism
\item rotation
\item type II deformation
\end{enumerate}
\end{theorem}
We will recall the mapping torus construction and introduce  (anti-)rota\-tion and type II deformation of a metric $f$-($K$-)contact or $S$-manifold in Section \ref{sec:deformations}. As in the Sasakian setting \cite{BoyerGalicki}, type II deformation does not modify the characteristic vector fields, while it changes the one-forms of the structure by the addition of other one-forms. Rotation and anti-rotation do not have an analogue in the contact setting: they are a type of deformation in which one applies an appropriate (constant) base change to the characteristic vector fields.

\section{Preliminaries}\label{sec:preliminaries}

Let us review the notion of a metric $f$-contact manifold. We assume that we are given linearly independent one-forms $\eta_1,\ldots, \eta_s$ on a smooth manifold $M^{2n+s}$, as well as vector fields $\xi_1,\ldots,\xi_s$  satisfying $\eta_i(\xi_j)=\delta_{ij}$. In this setting the tangent bundle $TM$ decomposes as the sum of the parallelizable subbundle given by the span of the $\xi_i$, and the intersection of the kernels of the one-forms $\eta_i$. In addition, let $f$ be a $(1,1)$-tensor on $M$ satisfying 
\[
f(\xi_i)=0,\qquad \im f = \bigcap_{i=1}^s \ker(\eta_i),\qquad \left.f^2\right|_{\im f} = \left.-\id\right|_{\im f}.
\]
Note that $f$ is then of constant rank and satisfies $f^3+f=0$, i.e., it defines an $f$-structure in the sense of Yano \cite{yano}. For $s=0$ one recovers the notion of an almost complex, and for $s=1$ that of an almost contact manifold.

A Riemannian metric $g$ on $M$ satisfying 
\[
g(fX,fY) = g(X,Y) - \sum_{i=1}^s \eta_i(X)\eta_i(Y)
\]
is called \emph{compatible} with the $f$-structure, and in this situation one speaks of a \emph{metric $f$-manifold} $(M,f,\eta_i,\xi_i,g)$. The \emph{fundamental $2$-form} of the metric $f$-manifold $M$ is given by
\[
\omega(X,Y) = g(X,fY),
\]
for $X,Y\in TM$. One calls $M$ a \emph{metric $f$-contact manifold} if 
\[
d\eta_i = \omega
\]
for all $i$.

Let $(M,f,\xi_i,\eta_i,g)$ be a metric $f$-contact manifold. We will refer to the vector fields $\xi_i$ as the \emph{characteristic vector fields} of the structure, and denote by $\cF$ the \emph{characteristic foliation}, i.e., the foliation spanned by the characteristic vector fields. We recall that by \cite[Equation (2.4)]{cabrerizo-fern}, the characteristic vector fields commute, i.e., 
\begin{equation}\label{ab}
 [\xi_i,\xi_j]=0.
\end{equation}
By \cite[Theorem 2.6]{cabrerizo-fern} $\xi_i$ is Killing if and only if 
\begin{equation}\label{eq:fKLxi}
{\mathcal L}_{\xi_i} f=0.
\end{equation} We call $M$ a \emph{metric $f$-$K$-contact manifold} if all characteristic vector fields are Killing vector fields.

A metric $f$-contact manifold satisfying the \emph{normality} condition
\begin{equation*}
 [f,f]+2\summ d\etaa \otimes \xia =0,
\end{equation*}
is called an \emph{$S$-manifold}; here $[f,f]$ denotes the Nijenhuis torsion of $f$, i.e.,
\[
 [f,f](X,Y)=f^2[X,Y]+[fX,fY]-f[X,fY]-f[fX,Y], 
\]
where $X,Y$ are arbitrary vector fields on $M$.
By \cite[Theorem 1.1]{blair70}, the characteristic vector fields of an $S$-manifold are Killing, i.e., $S$-manifolds are metric $f$-$K$-contact.

\section{Deformation of metric $f$-$K$-contact manifolds}\label{sec:deformations}
In this section we describe four ways to construct new metric $f$-$K$-contact manifolds out of old ones. We explain how the well-known type II deformation generalizes from the Sasakian setting to metric $f$-$K$-contact manifolds, introduce new constructions called \emph{rotation} and \emph{anti-rotation} that do not exist for Sasakian manifolds, and recall the construction of the mapping torus of a metric $f$-$K$-contact manifold from \cite{gl}.

We begin with rotation and anti-rotation. In these deformations, the characteristic foliation $\cF$ and the $(1,1)$-tensor $f$ remain the same; essentially, one applies an appropriate (constant) base change to the characteristic vector fields $\xi_i$.

\begin{lemma}\label{rotation}
 Let $(M,f,\xi_1,\dots,\xi_{s}, \eta_1,\dots,\eta_{s},g)$ be a metric $f$-($K$-)contact manifold (resp.\ an $S$-manifold), and $A=(a_{ij})\in {\mathrm{O}}(s)$ an orthogonal $s\times s$ matrix, such that $c_i:=\sum_{j=1}^s a_{ij}\neq 0$ for every $i\in \{1,\dots ,s\}$. Then, the tensor fields
\[
 {\eta}'_i:=\sum_{t=1}^s  a_{ti}c_t \,\eta_t, \quad {\xi}'_i:= \sum_{t=1}^s \frac{1}{c_t}a_{ti}\,\xi_t, \quad i\in \{1,\dots ,s\},
\]
together with
 \begin{equation}\label{gfr}
 \begin{aligned}
 &{g}':=g-\sum_{\alpha=1}^s\etaa \otimes \etaa +\sum_{\alpha=1}^s{\eta}'_{\alpha} \otimes {\eta}'_{\alpha}, \quad {f}':=f,
 \end{aligned}
\end{equation} 
determine a new metric $f$-($K$-)contact structure (resp.\ an $S$-structure) on $M$, that we call a \emph{rotation} of $(f,\xi_i, \eta_i,g)$.

\vspace{0.2 cm}

The tensor fields
\[
 \tilde{\eta}_i:=\frac{1}{c_i}\sum_{t=1}^s  a_{it} \,\eta_t, \quad \tilde{\xi}_i:= {c_i}\sum_{t=1}^s a_{it}\,\xi_t, \quad i\in \{1,\dots ,s\},
\]
 \begin{equation}\label{gf}
 \begin{aligned}
 &\tilde{g}:=g-\sum_{\alpha=1}^s\etaa \otimes \etaa +\sum_{\alpha=1}^s\tilde{\eta}_{\alpha} \otimes \tilde{\eta}_{\alpha}, \quad\tilde{f}:=f,
 \end{aligned}
\end{equation} 
determine a new metric $f$-($K$-)contact structure (resp.\ an $S$-structure) on $M$, that we call an \emph{anti-rotation} of $(f,\xi_i, \eta_i,g)$. 
\end{lemma}

\begin{proof}
We prove only the statements for the rotation; the result for the anti-rotation is entirely analogous.  From \eqref{gfr} it follows that the fundamental $2$-form ${\omega}'$ of $({f}', {\xi}'_i, {\eta}'_i,{g}')$ coincides with the fundamental $2$-form $\omega$ of $(f,\xi_i, \eta_i,g)$, and that 
 \begin{align*}
  {g}'({f}'X,{f}'Y)&=g(fX,fY)=g(X,Y)-\sum_{\alpha=1}^s\etaa (X)  \etaa(Y)\\
 &={g}'(X,Y)- \sum_{\alpha=1}^s{\eta}'_{\alpha}(X) {\eta}'_{\alpha}(Y)
 \end{align*}
for all vector fields $X,Y$ on $M$. Moreover, for all $i,k\in \{1,\dots ,s\}$,

\begin{equation*}
 {\eta}'_i({\xi}'_j)
 =\sum_{t,k=1}^s a_{kj} a_{ti}c_t \frac{1}{c_k} \,\delta_{tk}=\delta_{ij},\quad d {\eta}'_i=\sum_{t,j=1}^s  a_{ti} a_{tj} d {\eta}_j=\sum_{j=1}^s \delta_{ij} d {\eta}_j = {\omega}',
\end{equation*}
\begin{align*}
 {g}'({\xi}'_i, {\xi}'_j)&=g({\xi}'_i, {\xi}'_j)-\sum_{\alpha=1}^s\etaa({\xi}'_i) \etaa ({\xi}'_j) +\delta_{ij}\\
 &= \sum_{t,k=1}^s \frac{a_{ti}a_{kj}}{c_tc_k} \delta_{tk} - \sum_{\alpha=1}^s \frac{a_{\alpha i}a_{\alpha j}}{c_\alpha^2} + \delta_{ij}=\delta_{ij},
\end{align*}
and
\begin{align*}
 -\id+\summ {\eta}'_\alpha\otimes{\xi}'_\alpha &= -\id+\summ \left(\sum_{t=1}^s  a_{t \alpha }c_t \,\eta_t\right)\otimes\left( \sum_{j=1}^s  \frac{1}{c_j}a_{j \alpha } \,\xi_j \right)\\
 &= -\id+ \sum_{t,j=1}^s \frac{c_t}{c_j}\left(\summ a_{t \alpha } a_{j \alpha } \right) \eta_t \otimes \xi_j\\
 &= -\id+ \sum_{t,j=1}^s \delta_{tj} \eta_t \otimes \xi_j=-\id+\sum_{j=1}^s \eta_j \otimes \xi_j =f^2= {f}'^2.
 \end{align*}
hence $({f}', {\xi}'_i, {\eta}'_i, {g}')$ is a metric $f$-contact structure.  
 Moreover, since
for every $i\in \{1,\dots,s\}$,
\begin{equation*}
\mathcal{L}_{{\xi}'_i}{f}'=\sum_{t=1}^s \frac{1}{c_t} a_{ti}\mathcal{L}_{{\xi}_t}{f},
\end{equation*}
and, because of $d\eta_\alpha' = \omega = d\eta_\alpha$ and $\sum_{i=1}^s \xi_i' = \sum_{i=1}^s \xi_i$, 
\begin{equation*}
 \,[{f}',{f}']+2 \summ d {\eta}'_\alpha\otimes {\xi}'_\alpha
  =[{f},{f}]+2 \sum_{j=1}^s  d {\eta}_j\otimes {\xi}_j,
\end{equation*}
the rotation $({f}', {\xi}'_i, {\eta}'_i, {g}')$ of a metric $f$-$K$-contact or of an $S$-structure on $M$ are metric $f$-contact structures of the same type.
\end{proof}

\begin{remark}\label{remark rotation}
Given a metric $f$-contact manifold $(M, f,\xi_i, \eta_i,g)$, the operations of rotation and anti-rotation with respect to an orthogonal matrix $A$ as in Lemma \ref{rotation} are inverse to each other. 
\end{remark}

We define now type II deformations in the context of metric $f$-contact manifolds. The following definition generalizes the corresponding definition from the Sasakian setting, see \cite{BoyerGalicki}, p.\ 240.

\begin{lemma} 
 Let $(M^{2n+s},f,\xi_1, \dots, \xi_s,\eta_1,\dots,\eta_s,g)$ be a metric $f$-($K$-)con\-tact (resp.\ $S$-) manifold. Let $\theta_1,\dots,\theta_s$ be closed, $\cF$-basic one-forms on $M$.  Then, the one-forms $$\bar{\eta}_i:=\eta_i+\theta_i, \; i\in\{1,\dots,s\},$$ 
together with the tensors 
$$\xi_1,\dots,\xi_s, \; \bar{g}:=g+\sum_{i=1}^s \eta_i \otimes \theta_i + \theta_i\otimes \bar{\eta}_i, \; 
\bar{f}=f-\sum_{i=1}^s (\theta_i \circ f) \otimes \xi_i,$$ 
determine a new metric $f$-$K$-contact (resp.\ $S$-)structure on $M$, that we call a \emph{type II deformation} of $(f,\xi_1, \dots, \xi_s,\eta_1,\dots,\eta_s,g)$. 
\end{lemma}

\begin{proof}
 Since the $1$-forms $\theta_1,\dots, \theta_s$ are closed and $\cF$-basic, 
 \[
 \bar{\eta}_i(\xi_j)=\delta_{ij}, \quad  d \bar{\eta}_i=d \eta_i=:\omega, \quad i,j\in \{1,\dots,s\}.
 \]
The kernel of the endomorphism $\bar{f}$ equals that of $f$, i.e., the span of the vector fields $\xi_i$, and 
\begin{align*}
\bar{f}^2X &= f(\bar{f}X)-\sum_{i=1}^s(\theta_i\circ f)(\bar{f}X)\xi_i
  ={f}^2X-\sum_{i=1}^s(\theta_i\circ f)(fX)\xi_i\\
  &=-X+\sum_{i=1}^s\eta_i(X)\xi_i+\sum_{i=1}^s\theta_i(X)\xi_i
  =-X+\sum_{i=1}^s\bar{\eta}_i(X)\xi_i
  \end{align*}
for every vector field $X$ on $M$. Moreover, for vector fields $X,Y$ on $M$, we compute
\begin{align*}
  \bar{g}(X,\bar{f}Y) &= {g}(X,\bar{f}Y)+\sum_{i=1}^s\left( \eta_i(X) \theta_i(\bar{f}Y) + \theta_i(X)\bar{\eta}_i(\bar{f}Y)\right)\\
  &= g(X,fY)-\sum_{j=1}^s\theta_j(fY)g(X,\xi_j)+\sum_{i=1}^s\eta_i(X) \theta_i({f}Y)\\
  &=g(X,fY)=\omega(X,Y)=d \bar{\eta}_i(X,Y)
\end{align*}
and
\begin{align*}
  \bar{g}(\bar{f}X,\bar{f}Y)&={g}(\bar{f}X,{f}Y)={g}({f}X- \sum_{i=1}^s\theta_i(fX)\xi_i,{f}Y)\\
  &={g}({f}X,{f}Y)=g(X,Y)-\sum_{i=1}^s\eta_i(X)\eta_i(Y)\\
  &=\bar{g}(X,Y)-\sum_{i=1}^s\left(\eta_i(X)\theta_i(Y)+\theta_i(X)\bar{\eta}_i(Y) \right)   -\sum_{i=1}^s\eta_i(X)\eta_i(Y)\\
  &=\bar{g}(X,Y)-\sum_{i=1}^s\bar{\eta}_i(X)\bar{\eta}_i(Y),
\end{align*}
and hence $(\bar{f},\xi_i,\bar{\eta}_i,\bar{g})$ is a metric $f$-contact structure on $M$. 
By definition of $\bar{f}$ we obtain
\begin{equation}\label{l}
 \mathcal{L}_{\xi_j}\bar{f}=\mathcal{L}_{\xi_j}{f}-\sum_{i=1}^s\mathcal{L}_{\xi_j}\left((\theta_i\circ f)\otimes\xi_i\right)
 =\mathcal{L}_{\xi_j}{f}-\sum_{i=1}^s\left(\mathcal{L}_{\xi_j}(\theta_i\circ f)\right)\otimes\xi_i.
\end{equation}
If $(f,\xi_i,\eta_i,g)$ is a metric $f$-$K$-contact structure, then $\mathcal{L}_{\xi_j}{f}=0$ for every $j\in \{1,\dots,s\}$ by Equation \eqref{eq:fKLxi}, and Equation~\eqref{l} becomes 
\begin{align*}
( \mathcal{L}_{\xi_j}\bar{f})X&=-\sum_{i=1}^s(\mathcal{L}_{\xi_j}(\theta_i\circ f))(X)\xi_i = \sum_{i=1}^s (-\mathcal{L}_{\xi_j}(\theta_i( fX))+\theta_i( f[\xi_j,X]))\xi_i\\
 &=\sum_{i=1}^s (-\mathcal{L}_{\xi_j}(\theta_i( fX))+\theta_i[\xi_j,fX])\xi_i= -\sum_{i=1}^s(\mathcal{L}_{\xi_j}\theta_i)(fX)\xi_i=0,
\end{align*}
where $X$ is any vector field on $M$. Then, by \cite[Theorem 2.6]{cabrerizo-fern}, $(\bar{f},\xi_i,\bar{\eta}_i,\bar{g})$ is a metric $f$-$K$-contact structure on $M$.

 By straightforward computations using Equation \eqref{ab} and the fact that the $\theta_i$'s are basic and closed, we obtain
\begin{equation}\label{f}
\begin{aligned}
 \,[\bar{f},\bar{f}](X,Y) = \;& [{f},{f}](X,Y)-\sum_{i=1}^s\theta_i([{f},{f}](X,Y))\xi_i\\
 &-\sum_{i=1}^s\theta_i(fX)(\mathcal{L}_{\xi_j}f)Y +\sum_{i=1}^s\theta_i(fY)(\mathcal{L}_{\xi_j}f)X .
\end{aligned}
\end{equation}
 So if we assume that $(f,\xi_i,\eta_i,g)$ is an $S$-structure on $M$, then $\mathcal{L}_{\xi_j}{f}=0$ (see Section~\ref{sec:preliminaries}), and from \eqref{f} we get
\begin{equation*}
\begin{aligned}
 \,[\bar{f},\bar{f}](X,Y)+2\sum_{i=1}^s d\bar{\eta}_i(X,Y){\xi}_i 
  = \, & [{f},{f}](X,Y)-\sum_{i=1}^s\theta_i([{f},{f}](X,Y))\xi_i \\
  &+2\sum_{i=1}^s d{\eta}_i(X,Y){\xi}_i =0;\\
\end{aligned}
\end{equation*}
hence $(\bar{f},\xi_i,\bar{\eta}_i,\bar{g})$ is again an $S$-structure on $M$.
\end{proof}

\begin{remark}
Type II deformation, rotation and anti-rotation do not change the characteristic foliation $\cF$ of a metric $f$-contact manifold.
\end{remark}

Rotations and type II deformations commute with each other in the sense of the following lemma. 

\begin{lemma} \label{lem:rottypeIIcommute}
 Let $(M,f,\xi_1,\dots,\xi_{s}, \eta_1,\dots,\eta_{s},g)$ be a metric $f$-contact manifold. Let $\theta_i$, $i=1,\ldots,s$, be closed, $\cF$-basic, one-forms on $M$, $A=(a_{ij})\in {\mathrm{O}}(s)$, and 
  $\tilde{\theta}_i:=\frac{1}{c_i}\sum_{k=1}^s a_{ik}\theta_k$, with $c_i:=\sum_{j=i}^s a_{ij}\neq 0$, $i\in \{ 1,\dots,s\}$. 
  
  The operations in (1) and (2) lead to the same metric $f$-contact structure on $M$:
 \begin{enumerate}
  \item Rotation of $(f,\xi_i,\eta_i,g)$ with respect to $A$ and then type II deformation with respect to $\theta_1, \dots,\theta_s$.
  \item Type II deformation of $(f,\xi_i,\eta_i,g)$ with respect to  $\tilde{\theta}_1, \dots,\tilde{\theta}_s$ and then rotation with respect to $A$. 
 \end{enumerate}

\end{lemma}
 \begin{proof}
  We denote by $(f',\xi'_i, \eta'_i, g')$ and by $(\bar{f},\bar{\xi}_i, \bar{\eta}_i,\bar{g})$ the metric $f$-contact structures on $M$ obtained from $(f,\xi_i, \eta_i,g)$ after performing  a rotation with respect to $A$, respectively a type II deformation with respect to $\tilde{\theta}_i$.
  A type II deformation of $(f',\xi'_i, \eta'_i, g')$ with respect to $\theta_1, \dots,\theta_s$ gives a new metric $f$-contact structure with the following structure tensors:
  \begin{equation*}
  \begin{aligned}
  \hat{\eta}_i&=\eta'_i+\theta_i, \quad \hat{\xi}_i=\xi'_i,\quad \hat{f}=f-\sum_{i=1}^s (\theta_i \circ f) \otimes \xi'_i,\\
  \hat{g} &=  g' +\sum_{i=1}^s \eta'_i \otimes \theta_i + \theta_i\otimes \hat{\eta}_i\\
          &=g-\sum_{i=1}^s\eta_i \otimes \eta_i +\sum_{i=1}^s\eta'_{i} \otimes \eta'_{i} +\sum_{i=1}^s \eta'_i \otimes \theta_i + \sum_{i=1}^s\theta_i\otimes (\eta'_i+\theta_i)\\
          &= g - \sum_{i=1}^s \eta_i\otimes \eta_i + \sum_{i=1}^s \hat{\eta}_i \otimes \hat{\eta}_i.
  \end{aligned}
  \end{equation*}
We check that the metric $f$-contact structure $(f'',\xi''_i, \eta''_i,g'')$ obtained from a rotation of $(\bar{f},\bar{\xi}_i, \bar{\eta}_i,\bar{g})$ with respect to $A$ coincides with $(\hat{f},\hat{\xi}_i, \hat{\eta}_i, \hat{g})$:
\begin{align*}
 \eta_i''&=\sum_{t=1}^s  c_t a_{ti} \bar{\eta}_t=\eta'_i +\sum_{t,k=1}^s  a_{ti} a_{tk}\theta_k=\eta'_i +\theta_i =\hat{\eta}_i, \quad \xi_i''=\hat{\xi}_i,\\
  \quad {f''} &= \bar{f}=f-\sum_{i=1}^s (\tilde{\theta}_i \circ f) \otimes {\xi}_i=f-\sum_{k=1}^s ({\theta}_k \circ f) \otimes \sum_{i=1}^s \frac{1}{c_i} a_{ik} {\xi}_i=\hat{f},\\
 {g''} &=  \bar{g} -\sum_{i=1}^s \bar{\eta}_i \otimes\bar{\eta}_i  + \sum_{i=1}^s{\eta}_i'' \otimes {\eta}_i'' \\
          &=g+\sum_{i=1}^s\eta_i \otimes \tilde{\theta}_i +\sum_{i=1}^s\tilde{\theta}_i \otimes \bar{\eta}_{i}-\sum_{i=1}^s \bar{\eta}_i \otimes \bar{\eta}_i +\sum_{i=1}^s \hat{\eta}_i\otimes \hat{\eta}_i=\hat{g}.
  \end{align*}
  \end{proof}
  \begin{remark}
The metric $f$-$K$-contact structure obtained from $(f,\xi_i, \eta_i,g)$ after a type II deformation with respect to $\theta_i+{\theta}'_i$ coincides with the metric $f$-$K$ contact structure obtained performing first a type II deformation with respect to $\theta$ of $(f,\xi_i, \eta_i,g)$ and then a type II deformation with respect to ${\theta}'_i$. In particular, the inverse operation of a type II deformation with respect to $\theta_i$ is a type II deformation with respect to $-\theta_i$.
\end{remark}
 
 \begin{remark}
 As the inverse operation of a rotation is an anti-rotation, cf.\ Remark \ref{remark rotation}, and the inverse operation of a type II deformation is again a type II deformation, anti-rotations and type II deformations commute in a similar way as in Lemma \ref{lem:rottypeIIcommute}.
 \end{remark}

 We recall from \cite[Section~3]{gl} how a metric $f$-structure induces in a natural way the same structure on the mapping torus of an automorphism. Let $(M,f,\xi_i,\eta_i,g)$ be a metric $f$-contact manifold and $\phi:M\rightarrow M$ an automorphism of the structure.
The tensors $(f,\xi_i,\eta_i,g)$ on $M$ induce the following natural metric $f$-contact structure on $M\times \R$:
 \begin{equation*}
 \begin{aligned}
  &\bar{f}(X)=f(X), \; \bar{f}\left(\frac{d}{dt}\right)=0, \bar{\eta}_{\alpha}(X)=\eta_{\alpha}(X), \; \bar{\eta}_{\alpha}\left(\frac{d}{dt}\right)=0,\\
  &\bar{\eta}_{s+1}(X)=\frac{1}{s}(\eta_1(X)+\dots +\eta_s(X)), \; \bar{\eta}_{s+1}\left(\frac{d}{dt}\right)=1,\\
 &\bar{\xi}_{s+1} :=\frac{d}{dt},\; \bar{\xi}_{\alpha} :=\xia-\frac{1}{s}\frac{d}{dt}, \quad \alpha = 1,\dots,s, \\
  \end{aligned}
 \end{equation*}
for each $X\in TM$ and where $\frac{d}{dt}$ denotes the standard coordinate vector field on $\mathbb{R}$,
 \begin{equation*}
\begin{aligned}
 & \bar{g}(X,Y)= g(X,Y), \; \bar{g}(X,\bar{\xi}_{\alpha})=0,\; \bar{g}(\bar{\xi}_{\alpha}, \bar{\xi}_{\beta})=\delta_{\alpha}^{\beta},
 \end{aligned}
\end{equation*}
for each  $X,Y \in \im(f)$ and ${\alpha}, {\beta} \in \{1,\dots,s+1\}$. This structure is invariant under the 
 $\Z$-action on $M\times \R$ determined by $\phi$:
 \[
  m \cdot (p,t) \mapsto (\phi^m(p),t+m t_0), 
 \]
where $t_0\in \R$, $t_0\neq 0$, and descends to the quotient of the $\Z$-action, i.e., the mapping torus $M_\phi$ of $(M,\phi)$, making it a metric $f$-contact manifold. In \cite{gl} we also computed that $M_\phi$ is a metric $f$-$K$-contact (resp.\ an $S$-)manifold if and only if $M$ is.

\section{Main result}
In this section we prove our main result, which states that any compact metric $f$-$K$-contact manifold is obtained from a compact $K$-contact manifold by successively applying rotations, type II deformations, and constructions of mapping tori. The main idea of the proof is to apply our cohomological splitting theorem for compact metric $f$-$K$-contact manifolds from \cite{gl} in order to find a suitable deformation (i.e., type II deformation, combined with anti-rotation) whose characteristic foliation has closed leaves, and then exhibit this deformed structure as a mapping torus.

\begin{lemma} \label{lemma H}
Let $M^{2n+s}$ be a compact connected metric $f$-$K$-contact manifold. Then
 \begin{equation*}
  H^1(M) = \mathrm{span}_{\R} \{[\eta_1-\eta_s],\ldots,[\eta_{s-1}-\eta_s]\}  \oplus H^1(M,\cF).
 \end{equation*}
 \end{lemma}
\begin{proof}
By \cite[Theorem~4.5]{gl} we have an isomorphism
\[
H^*(M) \equiv \Lambda(\R^{s-1}) \otimes H^*(M,\cF_{s-1}),
\]
where $\Lambda(\R^{s-1})$ embeds in $H^*(M)$ as the exterior algebra over the cohomology classes $[\eta_i-\eta_s]$. We thus obtain
\[
 H^1(M) = \mathrm{span}_{\R} \{[\eta_1-\eta_s],\ldots,[\eta_{s-1}-\eta_s]\}  \oplus H^1(M,\cF_{s-1})
 \]
 and are left with showing $H^1(M,\cF_{s-1}) = H^1(M,\cF)$.

To show this we make use of one of the exact sequences from \cite[Proposition~4.4]{gl}:
\begin{align*}
0 \longrightarrow H^1(M,\cF)\longrightarrow H^1(M,\cF_{s-1}) \longrightarrow H^{0}(M,\cF) \overset{\delta}\longrightarrow H^{2}(M,\cF)\longrightarrow \cdots ,
\end{align*}
where the connecting homomorphism $\delta$ is given by $\delta([\sigma])=[\omega\wedge \sigma]$. Since the fundamental $2$-form $\omega$ is a non-zero element of $H^{2}(M,\cF)$ \cite[Lemma 6.3]{gl}, the map $\delta: H^{0}(M,\cF)\rightarrow H^{2}(M,\cF) $ is injective and thus $ H^1(M,\cF)\cong H^1(M,\cF_{s-1})$.
\end{proof}
 
\begin{lemma}\label{lemma eta}
Let $(M,f,\xi_1,\dots,\xi_s, \eta_1,\dots,\eta_s,g)$ be a compact connected metric $f$-$K$-contact manifold, where $s\geq 2$. Then there exists an anti-rotation $(\tilde{f},\tilde{\xi}_i, \tilde{\eta}_i,\tilde{g})$ of $(f,\xi_i, \eta_i,g)$ and a closed, ${\cF}$-basic, $1$-form $\theta$ on $M$, such that the cohomology class $ [ \eta]$, where \[
 \eta:=\tilde{\eta}_s-\frac{1}{s-1}(\tilde{\eta}_1+\dots +\tilde{\eta}_{s-1})+\theta, 
\]
is a real multiple of an integer class. 

Moreover, the closed $1$-form $\eta$ is  nowhere vanishing and 
determines a codimension-one foliation $\cF_{\eta}:=\ker \eta$ with compact leaves.
\end{lemma}
 \begin{proof}
 We consider the open set $$U:=\{ A=(a_{ij})\in {\mathrm{O}}(s) \,|\, \sum_{t=1}^sa_{it}\neq 0, i=1,\dots,s \}$$ of ${\mathrm{O}}(s)$, and the map $h:U\rightarrow \R^s$, defined by
\[
h(A)= 
A^t \begin{pmatrix}                                                          
    \frac{1}{c_1(A)} & 0 & \dots & 0 \\
    0 & \frac{1}{c_2(A)} & \dots & 0 \\
    \vdots & \vdots & \ddots & \vdots \\
    0 & 0 & \dots & \frac{1}{c_s(A)}
  \end{pmatrix}   
       \begin{pmatrix}  
        -\frac{1}{s-1} \\ \vdots \\ -\frac{1}{s-1}\\ 1
       \end{pmatrix}.
\]
where $A^t$ is the transpose of the matrix $A$, and
 \[
  c_k: U\rightarrow \R; \quad (a_{ij}) \mapsto \sum_{j=1}^s a_{kj},
 \]
 for every $k\in \{1,\dots, s\}$.
 Observe that $h(A)$ gives the coordinates of the vector $$\eta:=\tilde{\eta}_s-\frac{1}{s-1}(\tilde{\eta}_1+\dots +\tilde{\eta}_{s-1}) \in \text{span}_{\R}\{\eta_1,\dots,\eta_s \},$$ with respect to the basis $\{\eta_1,\dots,\eta_s\}$, where $\tilde{\eta}_i=\frac{1}{c_i(A)}\sum_{t=1}^s  a_{it}\eta_t$ are the one-forms on $M$ obtained from $(f,\xi_i,\eta_i,g)$ after an anti-rotation with respect to $A$. Moreover, $ h$ maps to the codimension-one subspace of $\R^s$, 
 \[
  V:=\{ (u_1,\dots,u_s) \,|\, \sum_{i=1}^s u_i=0 \}\simeq \text{span}_{\R}\{\eta_1-\eta_s,\dots ,\eta_{s-1}-\eta_s\}.
 \]

To prove the first part of this lemma, by Lemma~\ref{lemma H}, it suffices to show that $\im {h}$ contains an open neighborhood of $\eta$ in $V$. For this purpose, we show that the image of the differential 
$(d {h})_{I}: \mathfrak{o}(s)\rightarrow \R^s $ contains $V$.
A direct computation shows that, for every $X\in \mathfrak{o}(s)$,
\begin{equation*}
 \begin{aligned}
  (d h)_{I}(X)                                                     
  =X^t \begin{pmatrix}  
        -\frac{1}{s-1} \\ \vdots \\ -\frac{1}{s-1}\\ 1
       \end{pmatrix}
       -\begin{pmatrix}                                                          
    {c_1(X)} & 0 & \dots & 0 \\
    0 & {c_2(X)} & \dots & 0 \\
    \vdots & \vdots & \ddots & \vdots \\
    0 & 0 & \dots & {c_s(X)}
  \end{pmatrix}   
       \begin{pmatrix}  
        -\frac{1}{s-1} \\ \vdots \\ -\frac{1}{s-1}\\ 1
       \end{pmatrix}.
 \end{aligned}
\end{equation*}
In particular, for every element of the orthogonal Lie algebra $\mathfrak{o}(s) $ of type
    \[ \tilde{X}=\begin{pmatrix}
        \begin{array}{c | c}
      \begin{array}{c c c}
       \bf{  0}
      \end{array} & \begin{array}{c c c} a_1 \\ \vdots \\a_{s-1}\end{array}
      \\
      \hline
       \begin{array}{c c c} -a_1& \dots &-a_{s-1}\end{array}& 0
     \end{array}
       \end{pmatrix},
    \]
    we have that
    \[
     (d h)_{I}(\tilde{X})=\left(\frac{1}{s-1}-1 \right) 
     \begin{pmatrix}  
        a_1 \\ \vdots \\ a_{s-1}\\ -\sum_{i=1}^{s-1}a_i
       \end{pmatrix}.
    \]
Then, $\im (dh)_I\supset V$. 

The second part of the lemma follows directly from \cite[Lemma 3.5]{bg}.   \end{proof}

\begin{proposition}\label{prop:onestep}
Let $M$ be a compact, connected, metric $f$-$K$-contact (resp.\ $S$-) manifold. Then, there exists a compact, metric $f$-$K$-contact (resp.\ $S$-) manifold $N$, so that $M$ is isomorphic to the mapping torus of $N$ with respect to an automorphism of the structure, up to a rotation and a type II deformation.
\end{proposition}

\begin{proof}
Let $(f,\xi_1,\dots,\xi_{s+1}, \eta_1,\dots,\eta_{s+1},g)$ be the metric $f$-$K$-contact structure on $M$, and consider the $1$-form
 \[
  \eta:= \eta_{s+1}-\frac{1}{s}\left(\eta_1+\dots +\eta_s \right),
 \] 
which is nowhere vanishing and closed, since $\eta(\xi_{s+1})=1$ and $d \eta_1=\dots =d \eta_{s+1}$.
By Lemma~\ref{lemma eta}, up to performing an anti-rotation and a type II deformation of the structure, 
we can assume that $[\eta]$ is a real multiple of an integer class, and
the leaves of the foliation $\cF_{\eta}$ defined by $\ker \eta$
are compact. We will show that in this situation $M$ is isomorphic to the mapping torus of a compact metric $f$-$K$-contact manifold $N$ with respect to an automorphism of $N$; as rotation and anti-rotation are inverse to each other by Remark \ref{remark rotation}, this will imply the proposition.

We observe that, for every $X\in \im f$, 
\[
 0=2d\eta(\xi_{s+1},X)=\xi_{s+1}(\eta(X))-X\eta(\xi_{s+1})-\eta[\xi_{s+1},X]=-\eta[\xi_{s+1},X];
\]
 thus $[\xi_{s+1},X]\in T \cF_{\eta}$ and, as the $\xi_i$ commute with each other, we obtain $[\xi_{s+1}, T \cF_{\eta}]\subset T \cF_{\eta}$. 
Then, by Proposition 2.2 of \cite{molino}, the flow $\phi$ of $\xi_{s+1}$ leaves $\cF_{\eta}$ invariant, namely it carries leaves of $\cF_{\eta}$ to leaves.

Let $N$ be any leaf of $\cF_{\eta}$. 
 We observe that there exists $t_0\in \R$ such that $\varphi:=\phi_{t_0}$ maps $N$ to itself: the map 
\[
 N\times \R\rightarrow M; \; (p,t)\mapsto \phi(p,t),
\]
is a local diffeomorphism and hence in particular $\phi(N\times \R)\subset M$ is an open, $\cF_{\eta}$-saturated subset. This implies that $\phi|_{N\times \R}:N\times \R\rightarrow M$ is surjective (otherwise $M\smallsetminus \phi(N\times \R)\neq \emptyset$ would be an open, $\cF_\eta$-saturated set, contradicting the fact that $M$ is connected). 
Thus, by the compactness of $M$, $\phi|_{N\times \R}:N\times \R\rightarrow M$ is not injective. Let $(p,t), (q,s)\in N\times \R$ such that $(p,t)\neq (q,s)$ and $\phi(p,t)=\phi(q,s)$. Since $\phi$ maps leaves to leaves and $\phi_{t-s}(p)=q$, we have that the map $\phi_{t-s}$ maps $N$ to $N$.

We construct on $N$ a metric $f$-$K$-contact structure such that $\varphi:N\rightarrow N$ is an automorphism of that structure. For every $i=1,\dots, s$, we denote by $\bar{\eta}_i$ the pullback one-form of $\eta_i$ on $N$, via the immersion $j:N\hookrightarrow M$, and by $\bar{\xi}_i$ the vector field on $N$, $j$-related to the vector field $ \xi_{i}+\frac{1}{s}\xi_{s+1}$ on $M$ tangent to $\cF_{\eta}$ (observe that $\eta\left(\xi_{i}+\frac{1}{s}\xi_{s+1}\right)=0$).
We define moreover a Riemannian metric $\bar{g}$ and a $(1,1)$-tensor $\bar{f}$ on $N$ by:
\begin{equation*}
 \begin{aligned}
 &\bar{g}(X,Y)=g(X,Y), \; \bar{g}(X,\bar{\xi}_i)=0, \; \bar{g}(\bar{\xi}_i,\bar{\xi}_j)=\delta_{ij},\\
 &\bar{f}(\bar{\xi}_i)=0, \;\bar{f}(X)=f(X),
 \end{aligned}
\end{equation*}
for every $i=1,\dots, s$, and $X,Y\in  \ker (\bar{\eta}_1)_p\cap \dots \cap \ker(\bar{\eta}_s)_p=(\im \bar{f})_p=(\im f)_p$, $p\in N$.
It is easy to check that $(\bar{f},\bar{\xi}_i,\bar{\eta}_i,\bar{g})$ is a metric $f$-$K$-contact structure on $N$, and the diffeomorphism $\varphi:N\rightarrow N$ preserves the tensors $\bar{\xi}_i$,  $\bar{\eta}_i$, $\bar{g}$, $\bar{f}$.
If $(f,\xi_1,\dots,\xi_{s+1}, \eta_1,\dots,\eta_{s+1},g)$ is an $S$-structure on $M$, then for all local vector fields $X,Y$ on $M$ tangent to $\cF_{\eta}$,
\begin{equation*}
\begin{aligned}
 ([\bar{f},\bar{f}]+2 \sum_{\alpha=1}^{s} d\bar{\eta}_{\alpha} \otimes \bar{\xi}_{\alpha})(X,Y) &=  [f,f](X,Y)+2 \sum_{\alpha=1}^{s} d{\eta}_{\alpha}(X,Y) \left({\xi}_{\alpha}+\frac{1}{s}\xi_{s+1}\right)\\
    &=    [f,f](X,Y)+2 \sum_{\alpha=1}^{s+1} d{\eta}_{\alpha}(X,Y) {\xi}_{\alpha}\\
    &=0,                                  
\end{aligned}
\end{equation*}
and hence $(N, \bar{f},\bar{\xi}_i,\bar{\eta}_i,\bar{g})$ is an $S$-manifold.

By \cite[Section 3]{gl}, and as recalled in Section \ref{sec:deformations}, the metric $f$-$K$-contact (resp.\ $S$-)structure on $N$ determines a metric $f$-$K$-contact (resp.\ $S$-)structure $(\tilde{f}, \tilde{\xi}_i, \tilde{\eta}_i, \tilde{g})$ on the mapping torus $N_{\varphi^{-1}}$.
Finally, observe that the map
\begin{equation*}
 \Psi: N_{\varphi^{-1}} \rightarrow M; \quad [(p,t)]\mapsto \phi(p,t),
\end{equation*}
is well defined, injective and a local diffeomorphism. Moreover, since $N_{\varphi^{-1}}$ and $M$ are compact, $\Psi$ is a diffeomorphism, which by construction preserves the given structure tensors: for every $X,Y\in \im f$, 
and $\alpha \in \{1,\dots ,s\}$,
\begin{equation*}
 \begin{aligned}
 &(d \phi)_{(p,t)}\tilde{\xi}_{\alpha}=(d \phi_t)_{p}({\xi}_{\alpha}+\frac{1}{s}{\xi}_{s+1})-\frac{1}{s}(d \phi_p)_{t}\left.\frac{d}{dt}\right|_t={\xi}_{\alpha}, \\
 &(d \phi)_{(p,t)}\tilde{\xi}_{s+1}=(d \phi_p)_{t}\left.\frac{d}{dt}\right|_t={\xi}_{s+1},\\
 &\tilde{g}(X,Y)=\bar{g}(X,Y)={g}(X,Y)=g((d \phi_t)_p X,(d \phi_t)_p Y)=(\phi^*g)(X,Y),\\
 &g((d \phi_t)_p X,(d \phi_t)_p \tilde{\xi}_i)=g((d \phi_t)_p X,{\xi}_i)=g(X, (d \phi_{-t})_p{\xi}_i)=0=\tilde{g}(X,\tilde{\xi}_i),\\
 &f((d \phi_t)_p X)=f(X)=\bar{f}(X)=\tilde{f}(X).
 \end{aligned}
 \end{equation*}
\end{proof}

\begin{theorem}
Any compact connected metric $f$-$K$-contact (resp.\ $S$-) manifold is obtained from a compact $K$-contact (resp.\ Sasakian) manifold by a finite iteration of the following operations:
\begin{enumerate}
\item construction of the mapping torus with respect to an automorphism
\item rotation
\item type II deformation
\end{enumerate}
\end{theorem}
\begin{proof}
This follows directly by applying Proposition \ref{prop:onestep} inductively, as in each step the dimension of the manifold decreases by one.
\end{proof}

\end{document}